\documentclass[a4paper,11pt]{article}
\usepackage[nointlimits]{amsmath}
\usepackage{amssymb,amsmath,amsthm}
\usepackage{amsfonts}
\usepackage{latexsym}
\usepackage{graphicx}
\usepackage{color}
\usepackage{layout}
\usepackage[english]{babel}
\def\HH{\overline{H}}
\def\NN{\mathbb N}

\newtheorem{thm}{Theorem}[section]
\newtheorem{lem}[thm]{Lemma}
\newtheorem{cor}[thm]{Corollary}

\begin{document}

\title{\bf More congruences from Ap\'ery-like formulae}

\author{{\sc Roberto Tauraso}\\
Dipartimento di Matematica\\
Universit\`a di Roma ``Tor Vergata'', Italy\\
{\tt tauraso@mat.uniroma2.it}\\
{\tt http://www.mat.uniroma2.it/$\sim$tauraso}
}

\date{}
\maketitle
\begin{abstract}
\noindent
We present some congruences modulo $p^{6-r}$
for sums of the type $\sum_{k=0}^{(p-3)/2}x^k{2k\choose k}/(2k+1)^r$,
for $r=1,2,3$ where $p>5$ is a prime.
This is a preliminary draft and more refined versions will follow.
\end{abstract}

\section{Introduction}

In proving the irrationality $\zeta(3)$,  Ap\`ery \cite{Ap:79} mentioned the identities
$$
\sum_{k=1}^{\infty}{1\over k^2}{2k \choose k}^{-1}={1\over 3}\, \zeta(2),\qquad
\sum_{k=1}^{\infty}{(-1)^k\over k^3}{2k \choose k}^{-1}=-{2\over 5}\, \zeta(3)
.$$
Later Koecher \cite{Ko:80}, and Leshchiner \cite{Le:81}, found
several analogous results for other $\zeta(r)$.
In particular, in \cite{Le:81}, the author gives several proofs of four identities.
Two of them, namely (3b) and (4b), are connected with the series expansion of odd powers of 
the complex function $\arcsin(z)$ (see \cite{BC:07}): for $|z|<2$ and for any odd integer $r>0$ then
$$
{\arcsin(z/2)^r}={r!\over 2}
\sum_{k=0}^{\infty}{{2k\choose k}\over 16^k}\cdot {z^{2k+1}\over (2k+1)}\cdot \HH_k(\{2\}^{r-1\over 2}).
$$
where 
$$\HH_n(a_1,a_2,\dots,a_d)=
\sum_{0\leq k_1<k_2<\dots<k_r<n}\; {1\over (2k_1+1)^{a_1}(2k_2+1)^{a_2}\cdots (2k_r+1)^{a_r}},$$
with $(a_1,a_2,\dots,a_r)\in (\NN^*)^r$.

By revisiting the combinatorial proof (due to D. Zagier) presented in Section 5 in \cite{Le:81}, it is easy to obtain
the finite versions of these identities: if $r$ is a positive odd integer then
\begin{align}\label{idodd}\nonumber
\sum_{k=0}^{n-1}{{2k\choose k}\over 16^k}&
\left(\sum_{j=0}^{{r-1\over 2}}{(-1)^j\HH_k(\{2\}^j)\over (2k+1)^{r-2j}}
-{(-1)^{{r-1\over 2}}\over 4}\cdot{\HH_k(\{2\}^{{r-1\over 2}})\over (2k+1)}\right)\\
&=\sum_{k=0}^{n-1}{(-1)^k\over (2k+1)^r}+{(-1)^{{r-1\over 2}}\over 4}
\sum_{k=0}^{n-1}{{2k\choose k}\over 16^k{n+k\choose 2k+1}}\cdot{(-1)^{n-k}\HH_k(\{2\}^{{r-1\over 2}})\over (2k+1)},
\end{align}
and if $r$ is a positive even integer then
\begin{align}\label{ideven}\nonumber
\sum_{k=0}^{n-1}{{2k\choose k}\over (-16)^k}&
\left(\sum_{j=0}^{{r\over 2}-1}{(-1)^j\HH_k(\{2\}^j)\over (2k+1)^{r-2j}}
+{(-1)^{{r\over 2}-1}\over 4}\cdot{\HH_k(\{2\}^{{r\over 2}-1})\over (2k+1)^2}\right)\\
&=\sum_{k=0}^{n-1}{1\over (2k+1)^r}+{(-1)^{{r\over 2}-1}\over 4}
\sum_{k=0}^{n-1}{{2k\choose k}\over (-16)^k{n+k\choose 2k+1}}\cdot{\HH_k(\{2\}^{{r\over 2}-1})\over (2k+1)^2}.
\end{align}

After some preliminary results, starting from \eqref{idodd} for $r=1$, and \eqref{ideven} for $r=2$,  that is
\begin{align}\label{iduno}
{3\over 4}\sum_{k=0}^{n-1}{{2k\choose k}\over 16^k(2k+1) }&=
\sum_{k=0}^{n-1}{(-1)^k\over (2k+1)}+{(-1)^n\over 4}
\sum_{k=0}^{n-1}{{2k\choose k}\over (-16)^k{n+k\choose 2k+1}}\cdot{1\over 2k+1},\\
\label{iddue}
{5\over 4}\sum_{k=0}^{n-1}{{2k\choose k}\over (-16)^k(2k+1)^2}&=
\sum_{k=0}^{n-1}{1\over (2k+1)^2}+{1\over 4}
\sum_{k=0}^{n-1}{{2k\choose k}\over (-16)^k{n+k\choose 2k+1}}\cdot{1\over (2k+1)^2},
\end{align}
we prove that for any prime $p>5$
\begin{align*}
\sum_{k=0}^{{p-3\over 2}}{{2k \choose k}\over 16^k(2k+1)}&\equiv
(-1)^{{p-1\over 2}}\left({H_{p-1}(1)\over 12}+{3\over 160}\,p^4 B_{p-5}\right)
\pmod{p^5},\\
\sum_{k=0}^{{p-3\over 2}}{{2k \choose k}\over (-16)^k(2k+1)^2}&\equiv
{H_{p-1}(1)\over 5p}+{7\over 20}\,p^3 B_{p-5}
\pmod{p^4},
\end{align*}
where $H_n(r)=\sum_{k=1}^{n}{1\over k^r}$ and $B_n$ denotes the $n$-th Bernoulli number.
In the last section we consider the case $r=3$.
These congruences confirm as conjectures in \cite[Conjecture 5.1]{ZWS:11}.
Notice that by taking the limit as $n$ goes to infinity in \eqref{iduno} and \eqref{iddue} we obtain
\begin{align*}
\sum_{k=0}^{\infty}{{2k \choose k}\over 16^k(2k+1)}&=
{4\over 3}\sum_{k=0}^{\infty}{(-1)^k\over (2k+1)}={4\over 3}\,\arctan(1)={\pi\over 3},\\
\sum_{k=0}^{\infty}{{2k \choose k}\over (-16)^k(2k+1)^2}&=
{4\over 5}\sum_{k=0}^{\infty}{1\over (2k+1)^2}={3\over 5}\,\zeta(2)={\pi^2\over 10},
\end{align*}
which allow us to evaluate the analogy between the finite and the infinite sum.
For more results of this flavour, involving  Ap\'ery-like formulae, see for example \cite{PP:11,MT:11,Ta:10,ZWS:11}.
In the recent preprint \cite{PP:11b}, the authors prove congruences modulo $p^{4-r}$ for sums of the general form
$\sum_{k=0}^{(p-3)/2}{2k\choose k}x^k/(2k+1)^r$  for $r=1,2$ in terms of the finite polylogarithms.  

\section{Results concerning multiple harmonic sums}
We define the {\it multiple harmonic sum} as
$$H_n(a_1,a_2,\dots,a_r)=
\sum_{0< k_1<k_2<\dots<k_r\leq n}\; {1\over k_1^{a_1}k_2^{a_2}\cdots k_r^{a_r}},$$
where $n\geq r>0$ and $(a_1,a_2,\dots,a_r)\in (\NN^*)^r$.
The values of many harmonic sums modulo a power of prime $p$ are well known. 
Here there is a list of results that we will need later. 

\begin{enumerate}

\item[(i)] (\cite[Theorem 5.1]{Sunzh:00}) for any prime $p>r+2$ we have
\begin{align*}
H_{p-1}(r)\equiv
\begin{cases}
-\frac{r(r+1)}{2(r+2)}\,p^2\,B_{p-r-2}  \pmod{p^3} &\mbox{if $r$ is odd,}\\
\frac{r}{r+1}\,p\,B_{p-r-1}  \pmod{p^2} &\mbox{if $r$ is even;}
\end{cases}
\end{align*}

\item[(ii)] (\cite[Theorem 3.1]{Zh:06}) for $r,s>0$, and for any prime $p>r+s$, we have
$$H_{p-1}(r,s)\equiv {(-1)^s\over r+s}{r+s\choose s} \,B_{p-r-s}\pmod{p};$$

\item[(iii)] (\cite[Theorem 3.5]{Zh:06}) for $r,s,t>0$, and for any prime $p>r+s+t$ such that
$r+s+t$ is odd, we have
$$H_{p-1}(r,s,t)\equiv {1\over 2(r+s+t)}\left((-1)^r{r+s+t\choose r}-(-1)^t{r+s+t\choose t}\right) \,B_{p-r-s-t}\pmod{p};$$

\item[(iv)] (\cite[Theorem 2.1]{Ta:10}) for any prime $p>5$, 
$$H_{p-1}(1)\equiv -{1\over 2}\,H_{p-1}(2)-{1\over 6}\,H_{p-1}(3)\pmod{p^5};$$

\item[(v)] (\cite[Lemma 3]{PP:11}) for any prime $p>5$, 
$$H_{p-1}(1,2)\equiv -3\,{H_{p-1}(1)\over p^2}+{1\over 2}\,p^2B_{p-5}\pmod{p^3};$$

\item[(vi)] (\cite[Theorem 5.2]{Sunzh:00}) for any prime $p>r+4$ we have
\begin{align*}
H_{{p-1\over 2}}(r)\equiv
\begin{cases}
-\frac{2^r-2}{r}\,B_{p-r}  \pmod{p} &\mbox{if $r>1$ is odd,}\\
{r(2^{r+1}-1)\over 2(r+1)}\,p\,B_{p-r-1}  \pmod{p^2} &\mbox{if $r$ is even.}
\end{cases}
\end{align*}
\end{enumerate}

The rest of this section is dedicated to proving some congruences involving $H_{{p-1\over 2}}(r)$ and $H_{{p-1\over 2}}(r,s)$.

\begin{lem}\label{Ldue} Let $r,a>0$ then for any prime $p>r+2$ 
\begin{align}\label{C1}
H_{p-1}(r)&\equiv H_{{p-1\over 2}}(r)+(-1)^r\sum_{k=0}^a{r-1+k\choose k} H_{{p-1\over 2}}(r+k)p^k\pmod{p^{a+1}},
\\\label{C3}\nonumber
H_{{p-1\over 2}}(r)&\equiv H_{\lfloor{p\over 4}\rfloor}(r)+
(-2)^r\sum_{k=0}^a{r-1+k\choose k} H_{{p-1\over 2}}(r+k)p^k\\
&\qquad -(-1)^r\sum_{k=0}^a{r-1+k\choose k}{H_{\lfloor{p\over 4}\rfloor}(r+k)\over 2^k}\, p^k \pmod{p^{a+1}}.
\end{align}
Moreover, if $r,s>0$ such that $r+s$ is odd then for any prime $p>r+s$
\begin{equation}\label{C2}
H_{{p-1\over 2}}(r,s)\equiv {B_{p-r-s}\over 2(r+s)}\left((-1)^s{r+s \choose s}+2^{r+s}-2\right)\pmod{p}.
\end{equation}
\end{lem}
\begin{proof} The congruence \eqref{C1} follows from the identity
$$H_{{p-1}}(r)=H_{{p-1\over 2}}(r)+\sum_{k=1}^{{p-1\over 2}}{1\over k^r(1-p/k)^r}.$$
Moreover the identity  
$$H_{{p-1\over 2}}(r)=
H_{\lfloor{p\over 4}\rfloor}(r)+2^r\left(
\sum_{k=1}^{(p-1)/2}{1\over (p-k)^r}-\sum_{k=1}^{\lfloor p/4\rfloor}({1\over (p-2k)^r}\right)$$
yields \eqref{C3}. Now we show \eqref{C2}. 
\begin{align*}
H_{{p-1}}(r,s)&=\!\!\!\!
\sum_{1\leq i<j\leq {p-1\over 2}}{1\over i^r j^s}
+\sum_{1\leq i,j\leq {p-1\over 2}}{1\over i^r (p-j)^s}
+\sum_{1\leq j<i\leq {p-1\over 2}}{1\over (p-i)^r (p-j)^s}\\
&\equiv
H_{{p-1\over 2}}(r,s)+(-1)^s\,H_{{p-1\over 2}}(r)H_{{p-1\over 2}}(s)+(-1)^{r+s}H_{{p-1\over 2}}(s,r)\\
&\equiv
H_{{p-1\over 2}}(r,s)-H_{{p-1\over 2}}(s,r) \pmod{p}.
\end{align*}
By the stuffle product property
$$H_{{p-1\over 2}}(r,s)+H_{{p-1\over 2}}(s,r)+H_{{p-1\over 2}}(r+s)=H_{{p-1\over 2}}(r)H_{{p-1\over 2}}(s)\equiv 0\pmod{p}.$$
Therefore
$$2H_{{p-1\over 2}}(r,s)\equiv H_{p-1}(r,s)-H_{{p-1\over 2}}(r+s)\pmod{p}$$
and by applying (ii) and (vi) we obtain \eqref{C3}. 
\end{proof}

\begin{thm}\label{T21} For any prime $p>2$
\begin{equation}\label{Wmezzo}
H_{{p-1\over 2}}(2)+{7\over 6}\, p\,H_{{p-1\over 2}}(3) +{5\over 8}\, p^2\,H_{{p-1\over 2}}(4)\equiv 0
\pmod{p^4}.
\end{equation}
\end{thm}
\begin{proof}  Let $m=\varphi(p^4)=p^3(p-1)$ and let $B_n(x)$ the $n$-th Bernoulli polynomial.
 For $r=2,3,4$, Faulhaber's formula implies
\begin{align*}
\sum_{k=1}^{(p-1)/2}k^{m-r}&={B_{m-r+1}\left({p-1\over 2}\right)-B_{m-r+1}\over m-r+1}\\
&= {B_{m-r+1}\left({p-1\over 2}\right)-B_{m-r+1}\left({1\over 2}\right)\over m-r+1}+{B_{m-r+1}\left({1\over 2}\right)-B_{m-r+1}\over m-r+1}\\
&= \sum_{k=r}^m{m-r\choose k-r}{B_{m-k}\left({1\over 2}\right)\over k-r+1}\cdot \left({p\over 2}\right)^{k-r+1}
+{B_{m-r+1}\left({1\over 2}\right)-B_{m-r+1}\over m-r+1}.
\end{align*}
By Euler's theorem, $H_{{p-1\over 2}}(r)\equiv \sum_{k=1}^{(p-1)/2}k^{m-r}\pmod{p^4}$.
Since $m$ is even, $B_{m-k}=0$ when $m-k>1$ and $k$ is odd. Moreover, $pB_{m-k}$ is $p$-integral and
$$B_{m-k}\left({1\over 2}\right)=(2^{1-m+k}-1)B_{m-k}\equiv (2^{1+k}-1)B_{m-k}\pmod{p^4}.$$
Hence
\begin{align*}
\sum_{r=2}^4 \alpha_r\, p^{r-2} H_{{p-1\over 2}}(r)&\equiv 
{p\over 2}\left(\alpha_2 -{6\over 7}\,\alpha_3\right)B_{m-2}\left({1\over 2}\right)\\
&\qquad +{p^3\over 8}\left(\alpha_2 -3\,\alpha_3+4\alpha_4\right)B_{m-4}\left({1\over 2}\right)
\pmod{p^4}.
\end{align*}
The right-hand side vanishes if we let $\alpha_2=1$, $\alpha_3=7/6$, and $\alpha_4=5/8$.
\end{proof}

\begin{cor} For any prime $p>5$
\begin{align}
\label{Hp2}
H_{{p-1}}(2)&\equiv -2\,{H_{{p-1}}(1)\over p}+{2\over 5}\,p^3 B_{p-5}\pmod{p^4},\\
\label{H2}
H_{{p-1\over 2}}(2)&\equiv -7\,{H_{{p-1}}(1)\over p}+{17\over 10}\,p^3 B_{p-5}\pmod{p^4},\\
\label{H3}
H_{{p-1\over 2}}(3)&\equiv 6\,{H_{{p-1}}(1)\over p^2}-{81\over 10}\,p^2 B_{p-5}\pmod{p^3},\\
\label{unoduepiuunotre}
H_{{p-1\over 2}}(1,2)+p\,H_{{p-1\over 2}}(1,3)
&\equiv -{9\over 2}\,{H_{p-1}(1)\over p^2} -{49\over 20}\,\, p^2 B_{p-5}
\pmod{p^{3}}.
\end{align}
\end{cor}
\begin{proof} The congruence \eqref{Hp2} follows from (iv) and (i).
By \eqref{Wmezzo} and (vi) we have 
$$H_{{p-1\over 2}}(2)+{7\over 6}pH_{{p-1\over 2}}(3)+{31\over 4}\,B_{p-5}\,p^3\equiv 0\pmod{p^4}.$$
From \eqref{C1} and (vi), we deduce
$$H_{p-1}(2)=2H_{{p-1\over 2}}(2)+2pH_{{p-1\over 2}}(3)+{66\over 5}\,B_{p-5}\,p^3\equiv 0\pmod{p^4}.$$
By solving these two congruences with respect to $H_{{p-1\over 2}}(2)$ and $H_{{p-1\over 2}}(3)$, and by using \eqref{Hp2}, we get the result.
Now we show \eqref{unoduepiuunotre}.  We consider the identity
$$H_{p-1}(1,2)=H_n(1,2)+H_n(1)\sum_{j=1}^n{1\over (p-j)^2} +\sum_{1\leq j<i\leq n}{1\over (p-i)(p-j)^2}$$
and by expanding the sums like in Lemma \ref{Ldue}, we get
\begin{align*}
H_{p-1}(1,2)&\equiv H_n(1,2)+H_n(1)(H_n(2)+2p\,H_n(3)+3p^2\,H_n(4))-H_n(2,1)\\
&\qquad -p\,(2H_n(3,1)+H_n(2,2))-p^2\,(3H_n(4,1)+2H_n(3,2)+H_n(2,3))
\pmod{p^{3}}.
\end{align*}
By applying the stuffle product for $H_n(1)H_n(2)$ and $H_n(1)H_n(3)$, and (i), (ii) and (vi), we obtain
$$
H_n(1,2)+p\,H_n(3,1)\equiv {1\over 2}\,H_{p-1}(1,2)-{1\over 2}\,H_n(3)-{27\over 4}\,p^2 B_{p-5}
\pmod{p^{3}}.$$
Thus the proof of \eqref{unoduepiuunotre} is complete as soon as we apply (v) and \eqref{H3}.
\end{proof}

\begin{thm} For any prime $p>5$
\begin{align}\nonumber
2(-1)^{{p-1\over 2}}\sum_{k=0}^{{p-3\over 2}}{(-1)^k\over 2k+1}
&\equiv   \HH_{{p-1\over 2}}(1) -p\HH_{{p-1\over 2}}(2)-p^2\HH_{{p-1\over 2}}(2,1)\\\label{alts}
&\quad+p^3\HH_{{p-1\over 2}}(2,2) +p^4\HH_{{p-1\over 2}}(2,2,1)   \pmod{p^{5}},                                                           
\\\label{altsb}
&\equiv -{1\over 2}\,H_{{p-1\over 2}}(1)+{11\over 16}\, H_{p-1}(1)-{57\over 1280}\, p^4 B_{p-5}
\pmod{p^{5}}.
\end{align}
\end{thm}
\begin{proof} The following identities hold
$$\sum_{k=0}^n{(-16)^k{n+k\choose 2k}\over (2k+1){2k\choose k}}
=2(-1)^{n}\sum_{k=0}^{n-1}{(-1)^k\over 2k+1}+{1\over 2n+1}\,
\quad\mbox{and}\quad
\sum_{k=0}^n{(-16)^k{n+k\choose 2k}\over (2k+1)^2{2k\choose k}}={1\over (2n+1)^2}.$$
Let $n=(p-1)/2$. Notice that for $k=0,\dots,n$
\begin{equation}\label{CCC}
{(-16)^k{n+k\choose 2k}\over {2k\choose k}}=\prod_{j=0}^{k-1}\left(1-{p^2\over (2j+1)^2}\right)=
\sum_{j=0}^{k-1} (-1)^jp^{2j}\,\HH_k(\{2\}^j).
\end{equation}
Thus the two identities yield
\begin{align*} 
\HH_n(1)-p^2\HH_n(2,1)+p^4\HH_n(2,2,1)&\equiv 2(-1)^{n}\sum_{k=0}^{n-1}{(-1)^k\over 2k+1}+{1-(-1)^n 4^{p-1}{2n\choose n}^{-1}\over p}\pmod{p^{6}},\\
\HH_n(2)-p^2\HH_n(2,2)+p^4\HH_n(2,2,2)&\equiv {1-(-1)^n 4^{p-1}{2n\choose n}^{-1}\over p^2}\pmod{p^{6}},
\end{align*} 
and, by subtracting $p$ times the second congruence from the first one, we get \eqref{alts}.

As regards \eqref{altsb}, let $m=\lfloor p/4\rfloor$, then
\begin{align*} 
(-1)^{n}\sum_{k=0}^{n-1}{(-1)^k\over 2k+1}&=
\sum_{k=1}^{n}{(-1)^{k}\over p-2k}
=\sum_{k=1}^{m}{1\over p-4k}-\sum_{k=1}^{m}{1\over p-(4k-2)}\\
&=2\sum_{k=1}^{m}{1\over p-4k}-\sum_{k=1}^{n}{1\over p-2k}
\equiv -{1\over 2}\sum_{k=0}^4 {H_m(k+1)\over 4^k}\,p^k+{1\over 2}\sum_{k=0}^4 {H_n(k+1)\over 2^k}\,p^k\pmod{p^{5}}.
\end{align*}
By taking a suitable linear combination of the equations \eqref{C3} for $r=1,2,4$ we find that
$$\sum_{k=0}^4 {H_m(k+1)\over 4^k}\,p^k
\equiv
{3\over 2}\,H_n(1)+{13\over 16}\,pH_n(2)+{1\over 2}\,p^2H_n(3)+{143\over 512}\,p^3 H_n(4)+{1\over 8}\,p^4H_n(5)
\pmod{p^{5}}.$$
Hence
$$(-1)^{n}\sum_{k=0}^{n-1}{(-1)^k\over 2k+1}\equiv
-{1\over 4}\,H_n(1)-{5\over 32}\,pH_n(2)-{1\over 8}\,p^2H_n(3)-{79\over 1024}\,p^3 H_n(4)-{1\over 32}\,p^4H_n(5)             
\pmod{p^{5}},$$
and \eqref{altsb} follows by applying \eqref{H2}, \eqref{H3}, and (vi). 
\end{proof}

Note that by using $\eqref{alts}$, $\eqref{altsb}$ and (iii), it can be shown that for any prime $p>5$,
$$H_{{p-1\over 2}}(2,1,2)\equiv 2\,H_{{p-1\over 2}}(1,2,2)\equiv 10\,H_{{p-1\over 2}}(2,2,1)\equiv -{15\over 4}\, B_{p-5} \pmod{p}.$$
Moreover, another consequence of the previous theorem is a generalization of Morley's congruence \cite{Ca:53}: for any prime $p>5$,
$${(-1)^{{p-1\over 2}}\over 4^{p-1}}{p-1\choose {p-1\over 2}}\equiv
1-{1\over 4}\,p\,H_{p-1}(1)-{1\over 80}\,p^5 B_{p-5}\pmod{p^6}.$$

\section{Proof of the main result}
We are finally ready to prove the congruences announced in the introduction.

\begin{thm}\label{TM} For any prime $p>5$ we have
\begin{align}\label{mc1}
\sum_{k=0}^{{p-3\over 2}}{{2k \choose k}\over 16^k(2k+1)}&\equiv
(-1)^{{p-1\over 2}}\left({H_{p-1}(1)\over 12}+{3\over 160}\,p^4 B_{p-5}\right)
\pmod{p^5},\\
\label{mc2}
\sum_{k=0}^{{p-3\over 2}}{{2k \choose k}\over (-16)^k(2k+1)^2}&\equiv
{H_{p-1}(1)\over 5p}+{7\over 20}\,p^3 B_{p-5}
\pmod{p^4}.
\end{align}
\end{thm}
\begin{proof} 
From \eqref{CCC}, it follows that
\begin{align}\label{conbin}\nonumber
{{2k\choose k}\over (-16)^k{n+k\choose 2k+1}}&={(2k+1){2k\choose k}\over (-16)^k(n-k){n+k\choose 2k}}
\equiv -{2\over \left(1-{p\over 2k+1}\right)\left(1-p^2\HH_k(2)+p^4\HH_k(4)-p^6\HH_k(6)\right)}\\\nonumber
&\equiv -2\left(1+{p\over 2k+1}+
\left({1\over (2k+1)^2}+\HH_k(2)\right)\,p^2+\left({1\over (2k+1)^3}+{\HH_k(2)\over 2k+1}\right)\,p^3
\right.\\
&\qquad\left.
\left({1\over (2k+1)^4}+{\HH_k(2)\over (2k+1)^2}+\HH_k(4)+\HH_k(2,2)\right)\,p^4
\right)\pmod{p^5}.
\end{align}
Therefore, by \eqref{iduno}, \eqref{conbin} and \eqref{alts}, we get
\begin{align}\label{a2}\nonumber
{3(-1)^n\over 2}\sum_{k=0}^{n-1}{{2k\choose k}\over 16^k(2k+1)}&\equiv
-2p\HH_{n}(2) -p^2(\HH_{n}(3)+2\HH_{n}(2,1))-p^3\HH_{n}(4)\\
&\qquad-p^4(\HH_{n}(5)+\HH_{n}(2,3)+\HH_{n}(4,1)))\pmod{p^5}.
\end{align}
Now we replace the terms $\HH_n$ with the corresponding expressions involving $H_{p-1}$ and $H_n$:
\begin{align*}
\HH_n(r)&=H_{p-1}(r)-{H_n(r)\over 2^r},\\
\HH_n(r,s)&=\sum_{0<j<i\leq n}{1\over (2(n-i)+1)^r(2(n-j)+1)^s}=
{1\over(-2)^{r+s}}\sum_{0<j<i\leq n}{1\over j^s i^r\left(1-{p\over 2i}\right)^r\left(1-{p\over 2j}\right)^s}\\
&\equiv {1\over(-2)^{r+s}}\left(H_n(s,r)+{p\over 2}\left(r H_n(s,r+1)+s H_n(s+1,r)\right)\right.\\
&\qquad\left.+{p^2\over 4}\left({r+1\choose 2} H_n(s,r+2)+rs H_n(s+1,r+1)+{s+1\choose 2} H_n(s+2,r)\right)\right)
\pmod{p^3}.
\end{align*}
So the right-hand side of \eqref{a2} becomes
\begin{align*}
&-2pH_{p-1}(2)-p^2H_{p-1}(3)-p^3H_{p-1}(4)-p^4H_{p-1}(5)\\
&\qquad +{p\over 2}H_n(2)+{p^2\over 8}(H_n(3)+2H_n(1,2))
+{p^3\over 16}\left(H_n(4)+4H_n(1,3)+2H_n(2,2)\right)\\
&\qquad-{p^4\over 32}\left(H_n(5)+4H_n(2,3)+3H_n(3,2)+7H_n(1,4)\right) \pmod{p^5}
\end{align*} 
Since 
$$H_n(2,2)={1\over 2}\left( H_n(2)^2-H_n(4)\right)\equiv -{31\over 5}\, pB_{p-5} \pmod{p^2},$$
by (i), (ii), and (vi), the above expression simplifies to
$$-2p H_{p-1}(2)+{1\over 2}\,pH_n(2)+{1\over 8}\,p^2 H_n(3)
+{1\over 4}\,p^2\left(H_n(1,2)+pH_n(1,3)\right)+{513\over 320}\,p^4 B_{p-5}
\pmod{p^5}.$$
Finally we apply \eqref{Hp2}, \eqref{H2}, \eqref{H3} and \eqref{unoduepiuunotre} 
$${H_{p-1}(1)\over 8}+{9\over 320}\,p^4 B_{p-5}\pmod{p^5}$$
which concludes our proof of \eqref{mc1}.

As regards \eqref{mc2}, by \eqref{iddue} and \eqref{conbin}, we have
\begin{align}\label{a1}\nonumber
{5}\sum_{k=0}^{n-1}{{2k\choose k}\over (2k+1)^2(-16)^k}&\equiv
4\HH_{n}(2)-2(\HH_{n}(2)+p\HH_{n}(3)+p^2(\HH_{n}(4)+\HH_{n}(2,2))\\
&\qquad +p^3(\HH_{n}(5)+\HH_{n}(2,3)))\pmod{p^4}.
\end{align}
As before, after replacing the terms $\HH_n$, the right-hand side of \eqref{a1} becomes
\begin{align*}
&2H_{p-1}(2)-2pH_{p-1}(3)-2p^2H_{p-1}(4)-2p^3H_{p-1}(5)\\
&\qquad -{1\over 2}H_n(2)+{p\over 4}H_n(3)+{p^2\over 8}\left(H_n(4)- H_n(2,2)\right)
+{p^3\over 16}\left(H_n(5)-2H_n(2,3)-H_n(3,2)\right)  \pmod{p^4}.
\end{align*}
By (i), (ii), and (vi), it simplifies to
$$2H_{p-1}(2)-{1\over 2}H_n(2)+{1\over 4}p H_n(3) +{9\over 4}p^3B_{p-5}     \pmod{p^4}$$
Finally we apply \eqref{Hp2}, \eqref{H2} and \eqref{H3} and we get
$${H_{p-1}(1)\over p}+{7\over 4}\,p^3B_{p-5} \pmod{p^4}.$$
\end{proof}

\section{The case $r=3$}
From \eqref{idodd} for $r=3$, we have
\begin{align}\label{idtre}
\sum_{k=0}^{n-1}{{2k\choose k}\over 16^k}\left({1\over (2k+1)^3 }
-{3\over 4}{\HH_k(2)\over (2k+1)}\right)&=
\sum_{k=0}^{n-1}{(-1)^k\over (2k+1)^3}-{(-1)^n\over 4}
\sum_{k=0}^{n-1}{{2k\choose k}\HH_k(2)\over (-16)^k{n+k\choose 2k+1}}\cdot{1\over 2k+1}.
\end{align}
Let $n=(p-1)/2$ and let $m=\lfloor p/4\rfloor$, then,
by taking a suitable linear combination of the equations \eqref{C3} for $r=3,5$ we find that
\begin{align*} 
(-1)^{n}\sum_{k=0}^{n-1}{(-1)^k\over (2k+1)^3}
&=2\sum_{k=1}^{m}{1\over (p-4k)^3}-\sum_{k=1}^{n}{1\over (p-2k)^3}\\
&\equiv -{1\over 32}\left(
H_m(3)+ {3H_m(3)\over 4}\,p+{3H_m(4)\over 8}\,p^2\right)\\
&\qquad+{1\over 8}\left(
H_n(3)+ {3H_n(3)\over 2}\,p+{3H_n(4)\over 2}\,p^2\right)\\
&=-{1\over 64}\,H_n(3)-{3\over 16}\,pH_n(4)-{189\over 512}\,p^2H_n(5)\\
&\equiv -{3\over 32}\,{H_{n}(1)\over p^2}
+{21\over 1280}\, p^2 B_{p-5}\pmod{p^{3}}.
\end{align*}
Moreover, acting as in Theorem \ref{TM}, we obtain that
\begin{align*} 
\sum_{k=0}^{n-1}{{2k\choose k}\HH_k(2)\over (-16)^k{n+k\choose 2k+1}}\cdot{1\over 2k+1}
&\equiv 2\HH_n(2,1)+2p\HH_n(2,2)\\
&\qquad +2p^2(\HH_n(2,3)+\HH_n(4,1)+2\HH_n(2,2,1))\\
&\equiv -{9\over 8}\,{H_{p-1}(1)\over p^2}
+{9\over 320}\, p^2 B_{p-5}\pmod{p^{3}}.
\end{align*}
Hence
\begin{equation}\label{mc3e21}
(-1)^n\sum_{k=0}^{n-1}{{2k\choose k}\over 16^k}\left({1\over (2k+1)^3}
-{3\over 4}{\HH_k(2)\over (2k+1)}\right)\equiv {3\over 16}\,{H_{p-1}(1)\over p^2}
+{3\over 320}\, p^2 B_{p-5}\pmod{p^{3}}.
\end{equation}
On the other hand, since $3$ does not divide $2n+1$, by \eqref{CCC}, we have
$$
0=\sum_{k=0}^{n-1}{n+k\choose 2k+1}{(-1)^k\over 2k+1}\equiv
\sum_{k=0}^{n-1}{{2k\choose k}\over 16^k(2k+1)}
\left(1-p^2\,\HH_k(2)+p^4\HH_k(2,2)\right)\pmod{p^{5}}.$$
Since (the proof will be given in the next version)
$$\sum_{k=0}^{n-1}{{2k\choose k}\HH_k(2,2)\over 16^k(2k+1)}\equiv 
(-1)^n {5\over 864}\, B_{p-5} \pmod{p},$$
by \eqref{mc1}, it follows that
\begin{equation}\label{mc21}
\sum_{k=0}^{n-1}{{2k\choose k}\HH_k(2)\over 16^k(2k+1)}
\equiv (-1)^n\left({H_{p-1}(1)\over 12p^2}+{53\over 2160}\,p^2 B_{p-5}\right) \pmod{p^3}.
\end{equation}
Finally, by \eqref{mc3e21}, we get
\begin{equation}\label{mc21}
\sum_{k=0}^{n-1}{{2k\choose k}\over 16^k(2k+1)^3}
\equiv (-1)^n\left({H_{p-1}(1)\over 4p^2}+{1\over 36}\,p^2 B_{p-5}\right) \pmod{p^3}.
\end{equation}
Note that the values of the corresponding infinite series are known (see \cite{PP:10}):
$$
\sum_{k=0}^{n-1}{{2k\choose k}\over 16^k(2k+1)^3}={7\pi^3\over 216}
\quad\mbox{and}\quad
\sum_{k=0}^{\infty}{{2k\choose k}\HH_k(2)\over 16^k(2k+1)}={\pi^3\over 648}.
$$


\end{document}